\documentclass[12pt]{amsart}
\usepackage{amsmath,amssymb,amscd,latexsym}

\def\PP{\mathbb P}

\newtheorem{theorem}{Theorem}[section]

\newtheorem{question}[theorem]{Question}

\newtheorem{${}$}[theorem]{${}$}

\begin{document}

\title{A remark on the Abel-Jacobi morphism \linebreak
       for the cubic threefold$^{\ast}$}

\author{Ze Xu}

\begin{figure}[b]

{\small
\begin{tabular}{ll}
$^{\ast}$This work was done when the author was visiting
 Institut de Math\'{e}matiques de Jussieu.
\end{tabular}}
\end{figure}

\begin{abstract}
Let $X$ be a smooth cubic threefold and $J(X)$ be its intermediate Jacobian. We show that there exists a codimension 2 cycle $Z$ on $J(X)\times X$ with $Z_{t}$ homologically trivial for each $t\in J(X)$, such that the morphism $\phi_{Z}: J(X)\rightarrow J(X)$ induced by the Abel-Jacobi map is the identity. This answers positively a question of Voisin in the case of the cubic threefold.
\end{abstract}

\maketitle

\section{Introduction}\label{introduction}

A classical theorem of Abel states that

\begin{theorem}\label{1.1}
Let $C$ be a smooth projective complex curve. Then each fiber of the Abel-Jacobi map
$$AJ_{C}: \text{Sym}^{d}C\rightarrow J(C)$$
is a projective space for all $d\geq g(C)$. Moreover, the induced morphism
$$\text{CH}_{0}(C)_{\text{num}}\rightarrow J(C)$$
is an isomorphism.
\end{theorem}

In particular, the geometry of the fibers of the Abel-Jacobi map for curves is well understood.

For higher dimensional varieties, the work of Bloch-Ogus \cite{2}, Bloch-Srinivas \cite{3}, Merkurjev-Suslin \cite{11} and Murre \cite{12} leads to the following theorem, which can be regarded as the higher dimensional generalization of the second assertion of Theorem \ref{1.1}.

\begin{theorem}(\cite{12}) \label{1.2}
Let $X$ be a smooth projective complex variety such that $\text{CH}_{0}(X)$ is supported on a curve. Then $CH^2(X)_{\text{hom}}=CH^2(X)_{\text{alg}}$ and
the Abel-Jacobi map induces an isomorphism
$$AJ_{X}: CH^2(X)_{\text{hom}}\rightarrow J(X):=H^{3}(X,\mathbb{C})/(F^{2}H^{3}(X)\oplus H^{3}(X,\mathbb{Z})).$$
\end{theorem}
In the present note, we will consider the case where $X$ is a rationally connected
threefold, so that $CH_0(X)=\mathbb{Z}$ is supported on a point and
$CH^2(X)=CH_1(X)$.

 Since the group $\text{CH}_{1}(X)_{\text{alg}}$ does not has the structure of an algebraic variety, one has to be careful when stating  that $AJ_{X}$ is \emph{algebraic}. In fact, $\text{CH}_{1}(X)_{\text{alg}}$ is an inductive limit of quotients of algebraic varieties by an equivalence relation,
  and to say that the morphism $AJ_{X}$ is algebraic means by definition that for any smooth projective variety $Y$ and any codimension 2 cycle $Z$ on $Y\times X$ with $Z_{y}\in\text{CH}^{2}(X)_{\text{hom}}$ for any $y\in Y$, the induced morphism $$\phi_{Z}: Y\rightarrow J(X), \ \phi_{Z}(y)=AJ_{X}(Z_{y}),$$ is a morphism of algebraic varieties, which will be called \emph{the Abel-Jacobi morphism}.

An important observation made by Voisin is that, despite the similarity between
 Theorem \ref{1.2} and Abel's theorem \ref{1.1}, there are substantial differences between 1-cycles on threefolds with small $\text{CH}_{0}$ and 0-cycles on curves, which she relates to  the geometry of the fibers of the Abel-Jacobi morphisms.

In fact, the following two questions are proposed in \cite{13}.

\begin{question}\label{q1}
Let $X$ be a smooth projective threefold such that $AJ_{X}: \text{CH}_{1}(X)_{\text{alg}}\rightarrow J(X)$ is surjective. Is there a codimension 2 cycle $Z$ on $J(X)\times X$ with $Z_{t}\in\text{CH}^{2}(X)_{\text{hom}}$ for any $t\in J(X)$ such that the Abel-Jacobi morphism
$$\phi_{Z}: J(X)\rightarrow J(X),\ \phi_{Z}(t)=AJ_{X}(Z_{t})$$
is the identity?
\end{question}

As remarked by Voisin, Question \ref{q1} has a positive answer if the Hodge conjecture holds true for degree 4 {\it integral} Hodge classes on $J(X)\times X$.

\begin{question}\label{q2}
For which threefolds $X$  is the following property satisfied?

There exist a smooth projective variety $Y$ and a codimension 2 cycle $Z$ on $Y\times X$ with $Z_{y}\in\text{CH}^{2}(X)_{\text{hom}}$ for any $y\in Y$, such that the Abel-Jacobi morphism $\phi_{Z}: Y\rightarrow J(X)$ is surjective with rationally connected general fiber.
\end{question}

It is known that Question \ref{q2} has a positive answer for smooth cubic threefolds \cite{9}, \cite{10} and smooth complete intersections of two quadrics in $\mathbb{P}^{5}$ \cite{4}. It was proved in \cite{13} that if Question \ref{q2} has a positive answer for $X$ and the intermediate Jacobian $J(X)$ admits a 1-cycle $\Gamma$ such that $\Gamma^{*g}=g!J(X)$ in $CH_g(J(X))=\mathbb{Z}$, where $g=\text{dim}J(X)$, then Question \ref{q1} also has a positive answer for $X$. In particular, if the intermediate Jacobian of $X$ is isomorphic to the Jacobian of a curve, then Question \ref{q1} has a positive answer for $X$ if Question \ref{q2} does. Therefore, Question \ref{q1} has a positive answer for smooth complete intersections of two quadrics in $\mathbb{P}^{5}$.

Unfortunately, there are very few rationally connected threefolds whose intermediate Jacobians are not Jacobians and for which the existence
of a cycle $\Gamma$ as above is known (this is a very classical question in the case
of general cubic threefolds, and is equivalent in this case to the algebraicity
of the so-called minimal class $\frac{\Theta^4}{4!}$ which is an integral Hodge class on $J(X)$). Question \ref{q1} needs therefore other approaches.

In this note we give a positive answer to Question  \ref{q1} for any smooth cubic threefold. For the properties of the intermediate Jacobian of the cubic threefold, see \cite{5}. The key point of our proof lies in the observation that the moduli space of stable sheaves of rank 2 with Chern numbers $c_{1}=0, c_{2}=2, c_{3}=0$ on a smooth cubic threefold is fine.

We will work over the complex number field $\mathbb{C}$.

\bigskip

\medskip\noindent
{\it Acknowledgements}: The author would like to thank Professor Claire Voisin gratefully for bringing to him this interesting question, as well as useful discussions and kind helps for the abbreviated French version. He also thanks Professor Baohua Fu gratefully for careful reading of the preliminary version and suggestions.

\section{The Main Result}

In this section, we state and prove the main result of this note.

\begin{theorem}\label{2.1}
Let $X$ be a smooth cubic threefold. Then there exists a codimension 2 cycle $Z$ on $J(X)\times X$ with $Z_{t}\in\text{CH}^{2}(X)_{\text{hom}}$ for any $t\in J(X)$, such that the induced Abel-Jacobi morphism $\phi_{Z}: J(X)\rightarrow J(X)$ is the identity.
\end{theorem}

We will need a sufficient condition for an open subset of the moduli space of stable sheaves on a smooth projective variety to be fine.

Let $X$ be a smooth projective variety. Recall that the Grothendieck group modulo numerical equivalence $K_{\text{num}}(X)$ is defined to be $K(X)/\equiv$, where two classes $x$ and $y$ in $K(X)$ are said to be numerically equivalent (notation $x\equiv y$), if the difference $x-y$ is contained in the radical of the quadratic form $$(a,b)\longmapsto\chi(a\cdot b)=\int_{X}\text{ch}(a)\text{ch}(b)\text{td}(X)$$
 (cf. \cite{7}). Now fix a class $c\in K_{\text{num}}(X)$. Let $P$ be the associated Hilbert polynomial, $M^{s}$ be the moduli space of stables sheaves on $X$ and $M(c)^{s}\subset M^{s}$ be the open and closed part parametrizing stable sheaves of numerical class $c$.

\begin{theorem}(\cite[Th.4.6.5]{8})\label{2.2}
If the greatest common divisor of all numbers $\chi(c\cdot\mathcal{F})$, where $\mathcal{F}$ runs through some collection of coherent sheaves on $X$, is equal to 1, then there is a universal sheaf on $M(c)^{s}\times X$.
\end{theorem}

\begin{theorem}\label{2.3}
Let $X$ be a smooth cubic threefold. Then the moduli space $M_{X}^{s}(2;0,2)$ of stable sheaves of rank 2 with Chern numbers $c_{1}=0, c_{2}=2, c_{3}=0$ on $X$ is fine. Equivalently, there exists a universal sheaf on $M_{X}^{s}(2;0,2)\times X$.
\end{theorem}

\begin{proof}
 Let $c$ be the numerical class of a locally free, stable sheaf $\mathcal{E}$ of rank 2 on $X$. Recall that $\text{Pic}(X)=\mathbb{Z}\cdot h$, where $h$ is the class of a hyperplane section of $X$. The Chow group of 1-cycles on $X$ modulo algebraic equivalence $A_{1}(X)=\mathbb{Z}\cdot l$, where $l$ is the class of a line on $X$. Note that $h^{2}\equiv3l$. Since $X$ is Fano, $\text{CH}_{0}(X)=\mathbb{Z}\cdot pt$. Then $\text{ch}(c)\equiv\text{ch}(\mathcal{E})=2-2l$. Since $c_{1}(\mathcal{T}_{X})=2h$ and $c_{2}(\mathcal{T}_{X})=12l$, $\text{td}(X)=1+h+2l+pt$. It is easy to see that $\text{ch}(\mathcal{O}_{X}(1))=1+h+\frac{3}{2}l$. Then we compute that $\chi(c\cdot\mathcal{E})=-4$, $$\chi(c\cdot\mathcal{O}_{X}(1))=\int_{X}\text{ch(c)}\cdot\text{ch}(\mathcal{O}_{X})\cdot\text{td}(X)$$
$$=\int_{X}(2-2l)\cdot(1+h+\frac{3}{2}l)\cdot(1+h+2l+pt)=5.$$
Obviously, the greatest common divisor of $\chi(c\cdot\mathcal{E})$ and $\chi(c\cdot\mathcal{O}_{X}(1))$ is equal to 1. Then Theorem \ref{2.2} implies that there exists a universal sheaf on $M_{X}^{s}(2;0,2)\times X$.
\end{proof}

{\it Proof of Theorem \ref{2.1}.}\quad
It is shown in \cite{6}, \cite{9} (see also \cite{1}) that the morphism $\phi: \text{Hilb}^{5t}(X)\rightarrow J(X)$ factorizes through the birational morphism $c_{2}: M_{X}(2;0,2)\rightarrow J(X)$. Moreover, letting $M_{0}$ be the open subset of $M_{X}(2;0,2)$ parametrizing locally free stable sheaves, the restricted morphism $M_{0}\rightarrow J(X)$ is an open immersion. Now we regard $M_{0}$ as an open subset of $J(X)$.  Let $Z'$ be the closure in $J(X)\times X$ of a global section of $\mathcal{E}|_{M'_{0}\times X}$, where $M'_0$ is an open subset of $M_0$ over which such a transverse section exists,   and $Z=Z'-J(X)\times C_{0}$, where $C_{0}$ is a quintic elliptic curve on $X$. Then the induced Abel-Jacobi morphism $\phi_{Z}: J(X)\dashrightarrow J(X)$ is the identity, since by construction it induces the natural inclusion on $M'_0$.

\vspace{1cm}

{\small

\noindent
{\bf Ze Xu}\\ Institute of Mathematics, Academy of Mathematics and Systems Science,
Chinese Academy of Sciences, Beijing 100190, China\\
{\bf Email: xuze@amss.ac.cn}

\end{document}